\documentclass[12pt]{article} 

\usepackage{mathptmx} 

\usepackage[
  margin=1in
]{geometry}

\usepackage[affil-it]{authblk}
\usepackage{latexsym}
\usepackage[reqno, namelimits, sumlimits]{amsmath}
\usepackage{amssymb,amsfonts,amsthm}
\usepackage{comment}
\usepackage{enumerate}
\usepackage{color}
\DeclareMathAlphabet{\mathcal}{OMS}{cmsy}{m}{n}

\theoremstyle{plain}

\newtheorem{theorem}{Theorem}[section]

\newtheorem{definition}[theorem]{Definition}
\newtheorem{lemma}[theorem]{Lemma}

\newtheorem{proposition}[theorem]{Proposition}

\theoremstyle{definition}
\newtheorem{remark}[theorem]{Remark}

\newcommand{\R}{\mathbb{R}}

\newcommand{\N}{\mathbb{N}}

\newcommand{\bP}{\mathbb{P}}
\newcommand{\bI}{\mathbf{I}}

\newcommand{\p}{\partial}

\newcommand{\norm}[1]{\lVert #1 \rVert}

\newcommand{\dto}{\downarrow}
\newcommand{\wto}{\rightharpoonup} 
\newcommand{\wstar}{\overset{\ast}{\rightharpoonup}}

\newcommand{\loc}{{\rm loc}}

\let\div\relax
\DeclareMathOperator{\div}{div}

\let\tilde\relas
\newcommand{\tilde}[1]{\widetilde{#1}}
\DeclareMathOperator*{\esssup}{ess\,sup}

\newcommand{\BMO}{{\rm BMO}}

\numberwithin{equation}{section}

\title{On local Type~I singularities of the Navier-Stokes equations and Liouville theorems}
\author{Dallas Albritton\footnote{School of Mathematics, University of Minnesota, Minneapolis, MN \\Email address: \texttt{albri050@umn.edu}}, Tobias Barker\footnote{DMA, {\'E}cole Normale Sup{\'e}rieure, CNRS, PSL Research University, 75005 Paris \\Email address: \texttt{tobiasbarker5@gmail.com}}}
\date{\today}

\begin{document}
\maketitle

\begin{abstract}
We prove that suitable weak solutions of the Navier-Stokes equations exhibit Type~I singularities if and only if there exists a non-trivial mild bounded ancient solution satisfying a Type~I decay condition. The main novelty is in the reverse direction, which is based on the idea of zooming out on a regular solution to generate a singularity.
By similar methods, we prove a Liouville theorem for ancient solutions of the Navier-Stokes equations bounded in $L^3$ along a backward sequence of times.
\end{abstract}

\section{Introduction}
In this paper, we consider potential singularities of the Navier-Stokes equations from the perspective of Liouville theorems. The main idea is to ``zoom in" on the singularity and classify the limiting objects. This approach is highly effective in the regularity theory of minimal surfaces~\cite{degiorgi65}, semilinear heat equations~\cite{gigakohnasymptoticallyselfsimilar}, harmonic maps~\cite{schoenuhlenbeckregularitytheory}, and many other PDEs.

Unlike the above examples, the three-dimensional Navier-Stokes equations have no known critical conserved quantities or monotonicity formulae. Because of this issue, Type~I conditions are typically imposed on the solutions; that is, we often ask that a critical quantity is finite near the singularity. For example, in the famous paper~\cite{escauriazasereginsverak}, Escauriaza, Seregin, and {\v S}ver{\'a}k demonstrated, via Liouville theorems, that $L^\infty_t L^3_x$ solutions do not form singularities. The axisymmetric case is exceptional because $r v_\theta$ satisfies a maximum principle, and in this case, Seregin and {\v S}ver{\'a}k proved that interior Type~I blow-up does not occur~\cite{sereginsverakaxisymmetric}. Liouville theorems were also used by Tsai in~\cite{tsai} and other authors (see~\cite{necasruzsverak,chaewolfasymptoticallyselfsimilar,guevaraphucselfsim}) to exclude self-similar singularities in quite general situations. However, many questions concerning feasible Type~I scenarios, e.g., discretely self-similar blow-up, remain completely open.  We refer to~\cite{sereginshilkinliouvillesurvey} for a recent survey of regularity results based on Liouville theorems.

A central object in the Liouville theory is the class of \emph{mild bounded ancient solutions}, which arise naturally as ``blow-up limits" of singular solutions (see~\cite{kochnadirashvili}).
These are defined to be solutions which satisfy the integral equation formulation of the Navier-Stokes equations and are bounded (in fact, smooth) for all backward times. The assumption that the solution is mild simply excludes the ``parasitic solutions" $v=\vec{c}(t)$, $q = -\vec{c} \,'(t) \cdot x$. At a conceptual level, classifying mild bounded ancient solutions serves to determine the possible model solutions on which a Navier-Stokes singularity must be based.

In~\cite{kochnadirashvili}, G. Koch, Seregin, {\v S}ver{\'a}k, and Nadirashvili conjectured that \emph{mild bounded ancient solutions are constant}.
Remarkably, the same authors proved that this is true in two dimensions and in the axisymmetric case without swirl (see Theorems 5.1-5.2 therein). A special case of the conjecture was recently verified by Lei, Ren, and Zhang in~\cite{lei2019ancient} when the solution is axisymmetric and periodic in the $z$-variable (see also~\cite{carrillo2018decay}). If true, the conjecture excludes Type~I singularities and implies that $D$-solutions of the steady Navier-Stokes equations are constant.\footnote{This is in contrast to the focusing semilinear heat equation $\p_t u - \Delta u = |u|^{p-1} u$, for which there is a non-trivial ground state whenever $p \geq p_c := \frac{n+2}{n-2}$.} If there is a counterexample to the conjecture, it is conceivable that it already occurs in the axisymmetric class. 

We are interested in a weak version of the above conjecture obtained by restricting to mild bounded ancient solutions having Type~I decay in backward time. With this modification, we can clarify the relationship between these solutions and Type~I singularities:




\begin{theorem}
\label{thm:localver}
The following are equivalent:
\begin{itemize}
	\item There exists a suitable weak solution with Type I singular point.
	\item There exists a non-trivial mild bounded ancient solution with $\bI < \infty$.
\end{itemize}
\end{theorem}

The relevant terminology will be defined below, as there is some sublety in the formulation of Type~I, see~\eqref{eq:typeinotion}. The quantity $\bI$ is defined in~\eqref{eq:idef}. For suitable weak solutions, see Definition~\ref{suitabledef}.

The main novelty of Theorem~\ref{thm:localver} is in the reverse direction. Our idea is to zoom out on an ancient (but regular) solution to generate a singular solution. This is known as the ``blow-down limit" in free boundary problems, and it has not yet been exploited in the Navier-Stokes literature.  Our primary tools are known and consist of estimates in Morrey spaces and the persistence of singularities introduced by Rusin and {\v S}ver{\'a}k in~\cite{rusinsver}. In principle, constructing ancient solutions with Type~I decay is a (difficult) route to obtaining Navier-Stokes singularities.

We will use a rather weak notion of Type~I in terms of the rescaled energy.
Let $z = (x,t) \in \R^{3+1}$, $Q(z,r) = B(x,r) \times ]t-r^2,t[$ be a parabolic ball, $Q' = Q(z,r)$, and
\begin{equation}
	A(Q') = \esssup_{t-r^2<t'<t} \frac{1}{r} \int_{B(x,r)} |v(x',t')|^2 \, dx',
\end{equation}
\begin{equation}
	C(Q') =  \frac{1}{r^2} \int_{Q(z,r)} |v|^3 \, dz',
\end{equation}
\begin{equation}
	D(Q') = \frac{1}{r^2} \int_{Q(z,r)} \left\lvert q-[q]_{x,r}(t') \right\rvert^{\frac{3}{2}} \, dz',
\end{equation}
\begin{equation}
	E(Q') = \frac{1}{r} \int_{Q(z,r)} |\nabla v|^2 \, dz'.
\end{equation}
If $\omega \subset R^{3+1}$ is open and $(v,q)$ is defined on $\omega$, then
\begin{equation}
	\label{eq:idef}
	\bI(\omega) = \sup_{Q' \subset \omega} A(Q') + C(Q') + D(Q') + E(Q')
\end{equation}
If $\omega$ is unspecified, we use $\omega = \R^3 \times \R_-$. Together, $v \equiv \text{const.}$ and $\bI < \infty$ imply $v \equiv 0$.

If $v$ is not essentially bounded in any parabolic ball centered at $z$, we say that $z$ is a \emph{singular point}.
Finally, if there exists a parabolic ball $Q'$ centered at the singular point $z$ and
\begin{equation}
	\label{eq:typeinotion}
	\bI(Q') < \infty,
\end{equation}
then we say that $z$ is a \emph{Type I} singularity.

Observe that~\eqref{eq:typeinotion} is adapted to the minimal requirements needed to make sense of the local energy inequality and partial regularity theory. In particular, $\bI(Q') \ll 1$ implies regularity. Our notion is also natural because it follows from boundedness of a variety of quantities considered to be Type~I in the literature, e.g.,
\begin{equation}
	\label{eq:TypeIinspacetime}
	\tag{a}
	\sup_{x,t} \left( |x^*-x| + \sqrt{T^*-t} \right) |v(x,t)|,
\end{equation}
\begin{equation}
	\label{eq:TypeIweakL3}
	\tag{b}
	\sup_{t} \norm{v(\cdot,t)}_{L^{3,\infty}},
\end{equation}
\begin{equation}
	\label{eq:TypeILp}
	\tag{$\text{c}_p$}
	\sup_t (T^*-t)^{\frac{1}{2}-\frac{3}{2p}} \norm{v(\cdot,t)}_{L^p},
\end{equation}
\begin{equation}
	\label{eq:TypeILinfty}
	\tag{$\text{c}_\infty$}
	\sup_{t} \sqrt{T^*-t} \norm{v(\cdot,t)}_{L^{\infty}},
\end{equation}
in the class of suitable weak solutions, see Lemma~\ref{lem:weakserrin}. In Theorem~\ref{thm:localvertypei}, we prove a version of Theorem~\ref{thm:localver} in the context of~\eqref{eq:TypeIinspacetime}-\eqref{eq:TypeILp} ($3 < p < \infty$) using Calder{\'o}n-type energy estimates, introduced in~\cite{calderon90}. Historically,~\eqref{eq:TypeILinfty} has been considered important, in part due to its success in the work of Giga and Kohn (see~\cite{gigakohnasymptoticallyselfsimilar}). However, an important distinction is that~\eqref{eq:TypeILinfty} is not well suited to the reverse direction, see Remark~\ref{rmkonpinfty}. Note that boundedness of one of~\eqref{eq:TypeIinspacetime}-\eqref{eq:TypeILinfty} is not known to imply boundedness of the other quantities.\footnote{Moreover, it does not appear to hold for other equations, e.g., the harmonic map heat flow or parabolic-elliptic Keller-Segel system in two dimensions. However, in the context of mild solutions, one may say that~\eqref{eq:TypeILp} for $p_1$ implies~\eqref{eq:TypeILp} for $p_2 \geq p_1$ (in a slightly smaller time interval), and in particular, implies~\eqref{eq:TypeILinfty}. Clearly,~\eqref{eq:TypeIinspacetime} implies~\eqref{eq:TypeIweakL3}. Of course, many more quantities are possible, e.g., space-time Lorentz norms, quantities involving the vorticity, quantities involving Besov spaces (see~\cite{sereginzhou2018}), etc.}
 \\

In this paper, we also prove a Liouville theorem for ancient solutions with Type~I decay along a backward sequence of times. The Liouville theorem of Escauriaza, Seregin, and {\v S}ver{\'a}k in~\cite{escauriazasereginsverak} states that an ancient suitable weak solutions in $L^\infty_t L^3_x$ vanishing identically at time $t=0$ must be trivial. It is natural to ask whether the condition on vanishing can be removed; is a mild ancient solution in $L^\infty_t L^3_x$ necessarily zero? Yes, since estimates of the form 
\begin{equation}
\label{eq:exampleest}
	\norm{v}_{L^\infty(Q(R/2))} \leq R^{-1} f \big( \norm{v}_{L^\infty_t L^3_x(Q(R))} \big)
\end{equation}
were considered by Dong and Du in~\cite{dongducritical}, 
where $f > 0$ is an increasing function. Hence, one may simply allow $R \to \infty$ in~\eqref{eq:exampleest}.\footnote{We thank Hongjie Dong for informing us of this proof. It is possible to prove~\eqref{eq:exampleest} using a compactness argument, persistence of singularities, and the local regularity result for $L^\infty_t L^3_x(Q)$ solutions in~\cite{escauriazasereginsverak}.
A similar Liouville theorem was proven in~\cite{sereginliouvillehalfspacel2infty} for ancient solutions in $L^\infty_t L^2_x(\R^2_+ \times ]-\infty,0[)$ by duality methods.} 
 On the other hand, the analogous result along a sequence of times is less obvious, and we prove it in the sequel:
\begin{theorem}
\label{thm:liouville}
If  $v$ is a mild ancient solution satisfying
\begin{equation}
	\sup_{k \in \N} \norm{v(\cdot,t_k)}_{L^3} < \infty
\end{equation}
 for a sequence of times $t_k \dto -\infty$, then
\begin{equation}
	v \equiv 0.
\end{equation}
\end{theorem}
The proof relies essentially on zooming out and the persistence of singularities, as in Theorem~\ref{thm:localver}. In this case, to control the solution, we use the theory of weak $L^{3,\infty}$ solutions developed in~\cite{sereginsverakweaksols,barkersereginsverakstability}, where $L^{3,\infty} = L^3_{\rm weak}$ is the Lorentz space/weak Lebesgue space. We prove a more quantitative version in Theorem~\ref{thm:moregeneralliouville}.
\\ 

Without Type~I assumptions, it is unclear what the existence of non-constant mild bounded ancient solutions says about the regularity theory. For example, the one-dimensional viscous Burgers equation is easily seen to be regular, but it admits non-constant traveling wave solutions $f(x-ct)$.\footnote{One may obtain other mild bounded ancient solutions of 1d viscous Burgers by solving the backward heat equation using a superposition of solutions $f(x_0) \exp\left[(x-x_0)+t\right]$ and applying the Cole-Hopf transformation.} These solutions are easily upgraded to higher dimensions by writing $u(x,t) = f(x\cdot \vec{n} - ct)$. Regarding Navier-Stokes solutions,  as there are no non-constant mild bounded ancient solutions in two dimensions~\cite{kochnadirashvili}, no such ``upgrade" is possible. The analogous results in the half-space remain open.\footnote{The relevant literature includes~\cite{gigahsumaekawaplaner,sereginsverakrescalinghalfspace,barkersereginmbas,sereginliouvillehalfspacel2infty}. Since the writing of this paper, Seregin has shown an analogue of Theorem~\ref{thm:localver} in the half-space~\cite{seregintypeihalfspace}. We remark that the relationship between various formulations of Type~I is less clear in the half-space.}


\section{Preliminaries}
\label{sec:preliminaries}

In this section, we recall some known facts about suitable weak solutions. We refer to~\cite{escauriazasereginsverak,sereginnotes,sereginsverakaxisymmetric,sereginsverakhandbook} for a review of the partial regularity theory; in particular, \cite{sereginsverakaxisymmetric,sereginsverakhandbook} contain many excellent heuristics.

Let $z = (x,t) \in \R^{3+1}$, $r>0$, and $Q' = Q(z,r)$ a parabolic~ball. We also write $Q(r) = Q(0,r)$ and $Q = Q(1)$.
\begin{definition}[Suitable weak solution]
	\label{suitabledef}
	We say that $(v,q)$ is a \emph{suitable weak solution} in $Q'$ if
	\begin{equation}
	v \in L^\infty_t L^2_x \cap L^2_t H^1_x(Q') \text{ and } q \in L^{\frac{3}{2}}(Q'),
	\end{equation}
	$(v,q)$ satisfies the Navier-Stokes equations on $Q'$ in the sense of distributions,
	\begin{equation}
	\left\lbrace
	\begin{aligned}
	\p_t v - \Delta v + v \cdot \nabla v + \nabla q &= 0 \\
	\div v &= 0
	\end{aligned}
	\right.
	\end{equation}
	 and $(v,q)$ satisfies the local energy inequality,
\begin{equation}
	\label{localenergyienq}
	\int_{B(x,r)} \zeta |v(y,t')|^2 \, dy  + 2 \int_{t-r^2}^{t'} \int_{B(x,r)} \zeta |\nabla v|^2 \, dy \, ds \leq $$ $$ \leq \int_{t-r^2}^{t'} \int_{B(x,r)} |v|^2 (\p_t + \Delta) \zeta + (|v|^2 + 2q) v \cdot \nabla \zeta \, dy \,ds,
\end{equation}
for all non-negative $\zeta \in C^\infty_0(B(x,r) \times ]t-r^2,t])$ and almost every $t' \in ]t-r^2,t]$.\footnote{By weak continuity in time, one may remove the ``almost every" restriction.}

	Finally, we say that $v$ is a suitable weak solution in $Q'$ (without reference to the pressure) if there exists $q \in L^{\frac{3}{2}}(Q')$ such that $(v,q)$ is suitable in $Q'$.
\end{definition}

The following lemma is proven in~\cite[Theorem 2.2]{lin}. The proof relies on the local energy inequality~\eqref{localenergyienq}, the embedding $L^\infty_t L^2_x \cap L^2_t H^1_x(Q) \hookrightarrow L^{\frac{10}{3}}(Q)$, and the Aubin-Lions lemma.
\begin{lemma}[Compactness]
\label{lem:compactnessforsuitable}
Let $(v^{(k)},q^{(k)})_{k \in \N}$ be a sequence of suitable weak solutions on $Q$ satisfying
\begin{equation}
	\sup_{k \in \N} \norm{v^{(k)}}_{L^3(Q)} + \norm{q^{(k)}}_{L^{\frac{3}{2}}(Q)} < \infty.
\end{equation}
Then there exists a suitable weak solution $(u,p)$ on $Q(R)$ for all $0 < R < 1$ such that
\begin{equation}
\label{eq:convergenceofsuitableweaksols}
	v^{(k)} \to u \text{ in } L^{3}_{\loc}(B \times ]-1,0]), \quad q^{(k)} \wto p \text{ in } L^{\frac{3}{2}}_{\loc}(B \times ]-1,0]),
\end{equation}
along a subsequence.
\end{lemma}

The next proposition is our primary tool. It is contained in Lemma 2.1 and Lemma 2.2 of~\cite{rusinsver}. However, as the statement therein is slightly different, we include a proof for completeness.
\begin{proposition}[Persistence of singularities]
	\label{pro:appearanceofsingularity}
	Let $(v^{(k)},q^{(k)})_{k \in \N}$ be a sequence of suitable weak solutions on $Q$ satisfying \eqref{eq:convergenceofsuitableweaksols}.
	If
	\begin{equation}
		\limsup_{k \to \infty} \; \norm{v^{(k)}}_{L^\infty(Q(R))} = \infty \text{ for all } 0 < R < 1,
	\end{equation}
	then
	\begin{equation}
	u \text{ has a singularity at the space-time origin}.
	\end{equation}
\end{proposition}
\begin{proof}
	We prove the contrapositive. Suppose that $u \in L^\infty(Q(R))$ for some $0<R<1$. Let $\epsilon > 0$ (to be fixed later). Then there exists $0 < R_0 < R$ (depending also on $\epsilon$) satisfying, for all $0 < r \leq R_0$,
	\begin{equation}
	\frac{1}{r^2} \int_{Q(r)} |u|^3 \, dx \, dt \leq \epsilon.
	\end{equation}
	This is because $u \in L^\infty(Q(R))$ is a subcritical assumption.
	Rescaling, we may set $R_0 = 1$. By the strong convergence in~\eqref{eq:convergenceofsuitableweaksols}, for $k$ sufficiently large (depending on $0 < r \leq 1$),
	\begin{equation}
	\label{vkest}
	\frac{1}{r^2} \int_{Q(r)} |v^{(k)}|^3 \, dx \, dt \leq 2\epsilon.
	\end{equation}
	We decompose the pressure as $q^{(k)} = \tilde{q}^{(k)} + h^{(k)}$, where
	\begin{equation}
		\label{qkdef}
		 \tilde{q}^{(k)} = (-\Delta)^{-1} \div \div (\varphi v^{(k)} \otimes v^{(k)}),
	\end{equation}
	$\varphi \in C^\infty_0(B)$ ($0 \leq \varphi \leq 1$) satisfies $\varphi \equiv 1$ on $B(3/4)$, and $h^{(k)}(\cdot,t)$ is harmonic in $B(3/4)$.
	 By~\eqref{qkdef} and Calder{\'o}n-Zygmund estimates,
	\begin{equation}
		\label{qkest}
		\frac{1}{r^2} \int_{Q(r)} |\tilde{q}^{(k)}|^{\frac{3}{2}} \,dx\,dt \leq \frac{C}{r^2}  \int_{Q(r)} |v^{(k)}|^3 \, dx \,dt.
	\end{equation}
	By the triangle inequality and~\eqref{qkest} (with $r=1$),
	\begin{equation}
		\label{hkunitscaleest}
		 \int_{Q} |h^{(k)}|^{\frac{3}{2}} \,dx \,dt \leq C  \sup_{l \in \N} \left( \int_{Q} |v^{(l)}|^{3} + |q^{(l)}|^{\frac{3}{2}} \, dx \,dt \right) \leq CM.
	\end{equation}
	($M > 0$ depends on $\epsilon$ through $R_0$.)
	By H{\"o}lder's inequality and interior regularity for harmonic functions, whenever $0 < r \leq 1/2$,
	\begin{equation}
	\label{hkscaleinvarest}
		\frac{1}{r^2} \int_{Q(r)} |h^{(k)}|^{\frac{3}{2}} \, dx \,dt \leq C r \int_{-\frac{1}{4}}^{0} \norm{ h^{(k)}(\cdot,t) }_{L^{\infty}(B(1/2))}^{\frac{3}{2}} dt\leq CMr,
	\end{equation}
	 Finally, one may combine~\eqref{vkest},~\eqref{qkest}, and~\eqref{hkscaleinvarest} and fix $\epsilon$ and $r$ sufficiently small to obtain
	 \begin{equation}
	\limsup_{k \to \infty} \frac{1}{r^2} \int_{Q(r)} |v^{(k)}|^3 + |q^{(k)}|^{\frac{3}{2}} \leq \epsilon_{\rm CKN},
	 \end{equation}
	 where $\epsilon_{\rm CKN} > 0$ is the constant in the $\epsilon$-regularity criterion. This ensures
	 \begin{equation}
	\limsup_{k \to \infty} \sup_{Q(r/2)} |v^{(k)}| \leq \frac{C_{\rm CKN}}{r},
	 \end{equation}
	 as desired.
\end{proof}

Since the forward direction of Theorem~\ref{thm:localver} deals with local solutions, it is useful to locally mimic the situation of the ``first singular time" in the Cauchy problem. The following proposition follows from partial regularity, see \cite[Lemma 3.2]{kukavicarusinzianeJMFM2017} and \cite[Theorem 3]{sereginsverakhandbook}.

\begin{proposition}[Regular cylinder lemma]
\label{pro:goodsingulartime}
	If $v$ is a suitable weak solution in $Q$ with singular point at the space-time origin, then there exist $z^* \in B(1/2) \times ]-1/4,0]$ and $0<R<1/2$ satisfying
	\begin{equation}
	v \in L^{\infty}(Q(z^*,R) \setminus Q(z^*,r)) \text{ for all } 0 < r < R.
	\end{equation}
\end{proposition}

It is possible to combine Proposition~\ref{pro:goodsingulartime} and Bogovskii's operator to truncate the solution, see~\cite{neustupapenel1999},~\cite[Remark 12.3]{taolocalizationcompactness}, and~\cite{albrittonbarkerlocalregII}. We will not require this here.


As discussed in the introduction, boundedness of other widely considered critical quantities is known to imply $\bI(Q') < \infty$. For example, this is true of the weak Lebesgue spaces:
\begin{lemma}[Weak Serrin implies Type~I]
	\label{lem:weakserrin}
	If $v$ is a suitable weak solution on $Q$ with
	\begin{equation}
	 v \in L^{q,\infty}_t L^{p,\infty}_x(Q),
	\end{equation}
	where $3 \leq p \leq \infty$ and $2 \leq q \leq \infty$ satisfy the Ladyzhenskaya-Prodi-Serrin condition
	\begin{equation}
	\label{eq:Ladyzhenprodi}
	\frac{3}{p} + \frac{2}{q} = 1,
	\end{equation}
	then, for all $Q' = Q(R)$ with $0 < R < 1$,
	\begin{equation}
	\label{eq:weakserrinimpliedbound}
	\bI(Q') < \infty.
	\end{equation}
\end{lemma}
Notice that having one of~\eqref{eq:TypeIinspacetime}-\eqref{eq:TypeILinfty} bounded is enough to apply Lemma~\ref{lem:weakserrin} (for suitable weak solutions).
It is already known that absolute smallness in the above $L^{q,\infty}_t L^{p,\infty}_x$ spaces (with the exception of the case $q = 2$) implies regularity, see~\cite{kimkozonoweakspaces}.


To prove Lemma~\ref{lem:weakserrin}, we use the critical Morrey-type quantities
\begin{equation}
	M^{s,l}(Q') = \frac{1}{R^{\kappa}} \int_{t-r^2}^t \left( \int_{B(x,r)} |v|^s \, dx' \right)^{\frac{l}{s}} \, dt',
\end{equation}
where $\kappa = l (2/l+3/s-1)$, defined for $1 \leq s,l \leq \infty$ (with the obvious modification when $l=\infty$).
The next lemma asserts that finiteness of rescaled energies $A,C,E$ (see~\cite{SereginCriticalMorrey2006}) or critical Morrey-type quantities $M^{s,l}$ (see~\cite[Theorem 6]{sereginsverakhandbook} and~\cite{SereginZajaczkowskiPOMI2006}) implies Type~I bounds for suitable weak solutions.
\begin{lemma}[Morrey-type estimates]
	\label{lem:Morreytypeestimates}
	Suppose $(v,q)$ is a suitable weak solution in $Q$ with
	\begin{equation}
	\min_{s,l} \left\lbrace \sup_{Q' \subset Q}  A(Q'), \sup_{Q' \subset Q}  C(Q'), \sup_{Q' \subset Q} E(Q'), \sup_{Q' \subset Q} M^{s,l}(Q') \right\rbrace < \infty,
	\end{equation}
	where $s>3/2$, $l>1$ are finite and required to satisfy\footnote{
	The statement in~\cite{sereginsverakhandbook} also contains the requirement $3/s+2/l-3/2 > \max \{ 2/l,1/2-1/l \}$. However, this requirement can be avoided by decreasing $s$ and/or $l$ using embeddings of Morrey spaces.}
	\begin{equation}
	\frac{3}{s} + \frac{2}{l} < 2.
\end{equation}
	Then, for all $Q' = Q(R)$ with $0 < R < 1$,
	\begin{equation}
	\label{eq:morreyIbound}
	\bI(Q') < \infty.
	\end{equation}
\end{lemma}
For the above result to hold, it is crucial that $(v,q)$ is already assumed to be suitable, since the proof relies on the local energy inequality. Indeed, the estimate which gives~\eqref{eq:morreyIbound} depends on the background quantities $C(1)$ and $D(1)$.

\begin{proof}[Proof of Lemma~\ref{lem:weakserrin}]
	Let $\delta > 0$ sufficiently small, so that $s=p-\delta$ and $l=q-\delta$ satisfy the requirements of Lemma~\ref{lem:Morreytypeestimates}. Then the embedding properties of Lorentz spaces imply
	\begin{equation}
	 M^{s,l}(Q')^{\frac{1}{l}} \leq C\norm{v}_{L^l_t L^s_x(Q')} \leq C\norm{v}_{L^{q,\infty}_t L^{p,\infty}_x(Q)},
	\end{equation}
	 for all parabolic balls $Q' \subset Q$.
\end{proof}

\section{Proof of Theorem~\ref{thm:localver}}

We now prove Theorem~\ref{thm:localver}. As the forward direction is essentially known, we focus on the reverse direction. The forward direction is also valid in the local setting with curved boundary without Type~I assumptions, see~\cite{albrittonbarkerlocalregII}.

\begin{proof}[Proof]
\textbf{Forward direction}. Suppose that $v$ is a suitable weak solution in $Q$ with singularity at the space-time origin and $\bI(Q) < \infty$. By Proposition~\ref{pro:goodsingulartime}, we may assume that $v \in L^{\infty}(Q \setminus Q(r))$ for all $0 < r \leq 1$. This may require considering an earlier singularity than the original. It is proven in~\cite[Theorem 2.8]{sereginsverakaxisymmetric} and~\cite[Section 5]{sereginsverakhandbook} that, under an appropriate rescaling procedure, such a solution (even without the Type~I assumption) gives rise to a non-trivial mild bounded ancient solution $u$. It is clear from the rescaling procedure in~\cite{sereginsverakaxisymmetric,sereginsverakhandbook} that $u$ will satisfy $\bI < \infty$.

\textbf{Reverse direction}. Suppose that $v$ is a non-trivial mild bounded ancient solution satisfying $\bI < \infty$. By translating in space-time as necessary, we have
\begin{equation}
	\label{eqN}
	\norm{v}_{L^\infty(Q)} = N > 0.
\end{equation}
Consider the sequence $(v^{(k)})_{k \in \N}$ of suitable weak solutions
	\begin{equation}
	\label{rescaling}
	v^{(k)}(x,t) = k v(k x, k^2 t), \quad (x,t) \in Q(2).
	\end{equation}
	By the uniform estimate
	\begin{equation}
	\label{eq:uniformest}
	\sup_{k \in \N} \bI(v^{(k)},Q(2)) < \infty
	\end{equation}
	and Lemma~\ref{lem:compactnessforsuitable}, there exists a subsequence and a suitable weak solution $(u,p)$ with
	\begin{equation}
	v^{(k)} \to u \text{ in } L^3(Q) \text{ and } q^{(k)} \wto p \text{ in } L^{\frac{3}{2}}(Q).
	\end{equation}
	Moreover,~\eqref{eqN} and~\eqref{rescaling} give
	\begin{equation}
	\norm{v^{(k)}}_{L^\infty(Q(1/k))} = k N \to \infty.
	\end{equation}
	Hence, Proposition~\ref{pro:appearanceofsingularity} implies that $u$ is singular at the space-time origin. Finally,
	\begin{equation}
	\bI(u) < \infty
	\end{equation}
	follows from~\eqref{eq:uniformest}. That is, the singularity is Type~I.
\end{proof}



We now address other formulations of Type~I.

\begin{theorem}
\label{thm:localvertypei}
Let $3 \leq p < \infty$.
The following are equivalent:
\begin{itemize}
	\item There exists a suitable weak solution in $Q$ with singularity at the space-time origin and
	\begin{equation}
		\label{eq:typei}
		\esssup_{-1<t<0} (-t)^{\frac{1}{2}-\frac{3}{2p}} \norm{v}_{L^{p,\infty}(B)} < \infty.
	\end{equation}
	\item There exists a mild bounded ancient solution satisfying
	\begin{equation}
	\label{eq:typeiancient}
		\esssup_{t<0} (-t)^{\frac{1}{2}-\frac{3}{2p}} \norm{v}_{L^{p,\infty}} < \infty.
	\end{equation}
\end{itemize}
\end{theorem}

\begin{remark}
\label{rmkonpinfty}
It is noteworthy that $p=\infty$ is omitted despite being a popular formulation of Type~I. This is because $\sup_{t < 0} \sqrt{-t} \norm{v(\cdot,t)}_{L^\infty} < \infty$ alone does not appear to guarantee $\bI < \infty$, or even that the local energy is finite up to (and including) the blow-up time. This is related to the fact that no global-in-time weak solution theory is known for $L^\infty$ initial data. However, the forward direction remains valid because Lemma~\ref{lem:weakserrin} implies $\bI(Q(1/2)) < \infty$ (with an estimate depending on the quantities $C(1)$ and $D(1)$ for suitable weak solutions).

When $p>3$, it is possible to prove Theorem~\ref{thm:localvertypei} with mild solutions replacing suitable weak solutions. One could also consider $\sup_{x,t} \left(|x| + \sqrt{-t} \right) |v|$, Lorentz spaces $L^{p,q}$ ($1 < q \leq \infty$), etc.
\end{remark}

\begin{proof}[Proof of Theorem~\ref{thm:localvertypei} (Reverse direction)]
Let $3 \leq p < \infty$. We allow the constants below to depend implicitly on $p$. It suffices to prove that a mild bounded ancient solution satisfying~\eqref{eq:typeiancient} also satisfies $\bI < \infty$. By translating in space-time and rescaling, we only need to demonstrate
\begin{equation}
	\label{eq:sufficestoshow}
	A(1/2) + C(1/2) + D(1/2) + E(1/2) \leq C(M),
\end{equation}
where
\begin{equation}
	\sup_{-1<t<0} (-t)^{\frac{1}{2} + \frac{3}{2p}} \norm{v(\cdot,t)}_{L^{p,\infty}} \leq M.
\end{equation}
We utilize a Calder{\'o}n-type splitting, see~\cite{calderon90,Jiasver2013,albrittonblowupcriteria}. Decompose $a := v(\cdot,-1) = \tilde{u_0} + \bar{u_0}$, where
\begin{equation}
	\tilde{u_0} = \bP \left( \mathbf{1}_{\{ |a| > \lambda M \}} a \right),
\end{equation}
and $\lambda > 0$ will be determined later. This gives
\begin{equation}
	\norm{\tilde{u_0}}_{L^2} \leq C(\lambda,M) \; \text{ and } \; \norm{\bar{u_0}}_{L^{2p}} \leq C_0(\lambda,M),
\end{equation}
where $C_0(\lambda,M) \to 0$ as $\lambda \to 0^+$.
We decompose the solution as
\begin{equation}
	\label{eq:vdecomp}
	v = V + U,
\end{equation}
where $V \in C([-1,0];L^{2p})$ is the mild solution of the Navier-Stokes equations on $\R^3 \times ]-1,0[$ with initial data $\bar{u_0}$. When $0 < \lambda \ll 1$ (depending on $M$),  $V$ is guaranteed to exist on $\R^3 \times ]-1,0[$, and
\begin{equation}
	\label{eq:Vest}
	\norm{V}_{L^\infty_t L^{2p}_x(\R^3 \times ]-1,0[)} + \norm{(t+1)^{\frac{1}{2}} \nabla V}_{L^\infty_t L^{2p}_x(\R^3 \times ]-1,0[)} \leq 1.
\end{equation}
 By the Calder{\'o}n-Zygmund estimates and pressure representation $Q = (-\Delta)^{-1} \div \div V \otimes V$,
\begin{equation}
	\label{eq:Qest}
	\norm{Q}_{L^\infty_t L^p_x(\R^3 \times ]-1,0[)} \leq C.
\end{equation}
The correction $U$ solves a perturbed Navier-Stokes equations with initial data $\tilde{u_0}$ and zero forcing term. It is possible to show that $U$ (which belongs to subcritical spaces) belongs to the energy space on $\R^3 \times ]-1,0[$ and satisfies the energy inequality. (There is standard perturbation theory involved, using that $v$ and $V$ are mild solutions, see~\cite{albrittonblowupcriteria} for details.) A Gronwall-type argument implies
\begin{equation}
	\label{eq:Uest}
	\norm{U}_{L^\infty_t L^2_x(\R^3 \times ]-1,0[)} + \norm{\nabla U}_{L^2(\R^3 \times ]-1,0[)} + \norm{U}_{L^{\frac{10}{3}}(\R^3 \times ]-1,0[)} \leq C(\lambda,M).
\end{equation}
Using $P = (-\Delta)^{-1} \div \div (U \otimes U + V \otimes U + U \otimes V)$, Calder{\'o}n-Zygmund estimates, and H{\"o}lder's inequality, we obtain
\begin{equation}
	\label{eq:Pest}
	\norm{P}_{L^{\frac{3}{2}}(Q)} \leq C(\lambda,M).
\end{equation}
Combining~\eqref{eq:vdecomp} with~\eqref{eq:Vest}-\eqref{eq:Pest} completes the proof of the reverse direction. We omit the proof of the forward direction.
\end{proof}

\section{Proof of Theorem~\ref{thm:liouville}}
We will now prove the Liouville theorem. In fact, we will prove the following, more quantitative generalization to the Lorentz space $L^{3,\infty}$. Let $\mathbb{B}$ denote the subspace of $\dot B^{-1}_{\infty,\infty}$ whose functions $f$ satisfy
\begin{equation}
	f(\lambda \cdot) \to 0 \text{ in the sense of distributions as } \lambda \to \infty.
\end{equation}

\begin{theorem}[Liouville theorem]
	\label{thm:moregeneralliouville}
	For all $M > 0$, there exists a constant $\epsilon = \epsilon(M) > 0$ satisfying the following property. Suppose that $v$ is a mild ancient solution\footnote{In this section, we consider mild solutions belonging to the class $L^\infty_{t,\loc} L^\infty_x(\R^3 \times ]-\infty,0[)$.} such that
	\begin{equation}
	\label{vboundedseqtimes}
	\norm{v(\cdot,t_k)}_{L^{3,\infty}} \leq M
	\end{equation}
	for a sequence $t_k \dto -\infty$. If
	\begin{equation}
	{\rm dist}_{\dot B^{-1}_{\infty,\infty}}(v(\cdot,0),\mathbb{B}) \leq \epsilon,
	\end{equation}
	then
		\begin{equation}
		\limsup_{k \to \infty} \sqrt{|t_k|/2} \norm{v}_{L^\infty(Q(\sqrt{|t_k|/2}))} < \infty.
	\end{equation}
	Hence,
	\begin{equation}
		v \equiv 0.
	\end{equation}
\end{theorem}

We will use the theory of weak $L^{3,\infty}$ solutions developed in~\cite{barkersereginsverakstability}. These are defined to be suitable weak solutions of the Navier-Stokes equations with initial data $u_0 \in L^{3,\infty}$ that additionally satisfy a decomposition $v = V + U$, where $V(\cdot,t) = S(t) u_0$ is the Stokes evolution of the initial data and $U$ belongs to the energy space with $\norm{U(\cdot,t)}_{L^2} \to 0$ as $t \to 0^+$. We will also use the following proposition, which is proven in~\cite{globalweakbesov} by contradiction and backward uniqueness arguments.

\begin{proposition}[Auxiliary proposition]
	\label{pro:auxiliarypro}
	For all $M>0$, there exists a constant $\epsilon_0 = \epsilon_0(M) > 0$ satisfying the following property. Suppose that $v$ is a weak $L^{3,\infty}$ solution on $\R^3 \times ]0,1[$ satisfying
	\begin{equation}
	\norm{v(\cdot,0)}_{L^{3,\infty}} \leq M
	\end{equation}
	and
	\begin{equation}
	\norm{v(\cdot,1)}_{\dot B^{-1}_{\infty,\infty}} \leq \epsilon_0.
	\end{equation}
	Then
	\begin{equation}
	v \text{ is essentially bounded in } \R^3 \times ]1/4,1[.
	\end{equation}
\end{proposition}
In fact, one may give pointwise bounds for $v$ on $\R^3 \times ]1/4,1[$, but this will not be necessary.

\begin{proof}[Proof of Theorem~\ref{thm:moregeneralliouville}]
	Suppose otherwise. That is, there exists a mild ancient solution $v$ satisfying
	\begin{equation}
	\norm{v(\cdot,t_{k})}_{L^{3,\infty}} \leq M
	\end{equation}
	for a sequence $t_k \dto -\infty$,
	\begin{equation}
	\label{eq:distance}
	{\rm dist}_{\dot B^{-1}_{\infty,\infty}}(v(\cdot,0),\mathbb{B}) \leq \epsilon_0/2,
	\end{equation}
	with $\epsilon_0 = \epsilon_0(M) > 0$ as in Proposition~\ref{pro:auxiliarypro}, and
	\begin{equation}
	\limsup_{k \to \infty} \sqrt{|t_{k}|/2} \norm{v}_{L^\infty(Q(\sqrt{|t_{k}|/2}))} = \infty.
	\end{equation}
Regarding \eqref{eq:distance}, we decompose $v(\cdot,0) = U + W$,
where $U \in \mathbb{B}$ and $\norm{W}_{\dot B^{-1}_{\infty,\infty}} \leq \epsilon_0$.

	We construct a sequence $(v^{(k)})_{k \in \N}$ of mild solutions on $\R^3 \times ]-1,0[$ by rescaling appropriately:
	\begin{equation}
		v^{(k)}(x,t) = \sqrt{|t_{k}|} v(\sqrt{|t_{k}|} x, |t_{k}| t).
	\end{equation}
	Since $v$ is mild, it is not difficult to show that $v^{(k)}$ is a weak $L^{3,\infty}$ solution on $\R^3 \times ]-1,0[$.
	Moreover,
	\begin{equation}
	\norm{v^{(k)}(\cdot,-1)}_{L^{3,\infty}} \leq M,
	\end{equation}
	\begin{equation}
	\label{eq:rescalingforvat0}
	v^{(k)}(\cdot,0) = \sqrt{|t_{k}|} v^{(k)}(\sqrt{|t_{k}|} \cdot,0) = U^{(k)} + W^{(k)},
	\end{equation}
	where $U^{(k)}$ and $W^{(k)}$ correspond to $U$ and $W$, appropriately rescaled,
	and
	\begin{equation}
	\label{eq:vkgrowing}
	\norm{v^{(k)}}_{L^\infty(Q(1/2))} \to \infty.
	\end{equation}
	Regarding~\eqref{eq:rescalingforvat0}, we find that $U^{(k)} \to 0$ in the sense of distributions and $\norm{W^{(k)}}_{\dot B^{-1}_{\infty,\infty}} \leq \epsilon_0$. Hence, there exists $W^\infty$ satisfying $\norm{W^\infty}_{\dot B^{-1}_{\infty,\infty}} \leq \epsilon_0$ and
	\begin{equation}
	v^{(k)}(\cdot,0) \to W^\infty \text{ in the sense of distributions}
	\end{equation}
	along a subsequence.

	Next, we recall a compactness result for the above sequence of weak $L^{3,\infty}$ solutions (see~\cite{barkersereginsverakstability,globalweakbesov}). There exists a weak $L^{3,\infty}$ solution $v^\infty$ on $\R^3 \times ]-1,0[$ and a subsequence such that
	\begin{equation}
		v^{(k)}(\cdot,-1) \wstar u(\cdot,-1) \text{ in } L^{3,\infty},
	\end{equation}
	where $\norm{u(\cdot,-1)}_{L^{3,\infty}} \leq M$,
	\begin{equation}
		v^{(k)} \to v^\infty \text{ in } L^{3}_{\loc}(\R^3 \times ]-1,0]),
	\end{equation}
	\begin{equation}
	q^{(k)} \wto q^\infty \text{ in } L^{\frac{3}{2}}_{\loc}(\R^3 \times ]-1,0]),
	\end{equation}
	and
	\begin{equation}
		v^{(k)}(\cdot,0) \to v^\infty(\cdot,0) \text{ in the sense of distributions}.
	\end{equation}
	In particular,
	\begin{equation}
		v^\infty(\cdot,0) = W^\infty.
	\end{equation}
	By Proposition~\ref{pro:auxiliarypro}, $v^\infty$ is essentially bounded in $\R^3 \times ]-3/4,0[$.

	We claim that $v^\infty$ has a singular point $z^* \in \overline{Q(1/2)}$. Indeed, due to~\eqref{eq:vkgrowing}, we have
	\begin{equation}
	\limsup_{k \to \infty} \norm{v^{(k)}}_{L^\infty(Q(z^*,R))} = \infty \text{ for all } 0 < R < 1/4,
	\end{equation}
	for some $z^* \in \overline{Q(1/2)}$, and we may invoke Proposition~\ref{pro:appearanceofsingularity}. This contradicts that $v^\infty$ is essentially bounded in $\R^3 \times ]-3/4,0[$ and completes the proof.
\end{proof}

We conclude with a few remarks:
\begin{remark}
The proof also implies that, if there exists a non-trivial mild ancient solution satisfying~\eqref{vboundedseqtimes}, then there exists a singular weak $L^{3,\infty}$ solution $v^\infty$ on $\R^3 \times ]-1,0[$. By considering the energy-class correction $u^\infty(\cdot,t) = v^\infty - S(t) v(\cdot,-1)$
after the initial time,
one obtains a singular weak Leray-Hopf solution with subcritical forcing term.

Using the theory of weak Besov solutions developed in~\cite{globalweakbesov}, similar statements hold when $L^{3,\infty}$ is replaced by $\dot{B}^{-1+\frac{3}{p}}_{p,\infty}$ and when $L^{3}$ is replaced by $\dot{B}^{-1+\frac{3}{p}}_{p,p}$ ($p>3$). While similar results remain unknown in $\BMO^{-1}$, a mild ancient solution satisfying $\norm{v(\cdot,t_k)}_{\BMO^{-1}} \to 0$ as $t_k \to -\infty$ must be identically zero. This follows from the perturbation theory in~\cite{kochtataru}.

Similar statements seem to hold \emph{mutatis mutandis} in the half-space with a different decomposition of the pressure, e.g., the one in~\cite{barkerser16blowup} (see also the weak $L^3(\R^3_+)$ solution theory developed in~\cite{Tuanthesis}). It is interesting to note that, in the half-space case, one has the option to zoom out on an interior or boundary point.
\end{remark}

\subsubsection*{Acknowledgments}
The authors would like to thank Gregory Seregin, Vladim{\i}r {\v S}ver{\'a}k, and Julien Guillod for helpful suggestions on a preliminary version of the paper. DA was supported by the NDSEG Graduate Fellowship and a travel grant from the Council of Graduate Students at the University of Minnesota.

\bibliographystyle{plain}
\bibliography{JMFMsubmission}

\end{document}